\documentclass{article} 

 \usepackage[usenames,dvipsnames]{pstricks}
 \usepackage{epsfig}
 \usepackage{pst-grad} 
 \usepackage{pst-plot} 
 \usepackage[space]{grffile} 
 \usepackage{etoolbox} 
 \makeatletter 
 \patchcmd\Gread@eps{\@inputcheck#1 }{\@inputcheck"#1"\relax}{}{}
 \makeatother

\usepackage[usenames,dvipsnames]{pstricks}
\usepackage{epsfig}
\usepackage{pst-grad} 
\usepackage{pst-plot} 
\usepackage[space]{grffile} 
\usepackage{etoolbox} 
\makeatletter 
\patchcmd\Gread@eps{\@inputcheck#1 }{\@inputcheck"#1"\relax}{}{}
\makeatother

 \usepackage[usenames,dvipsnames]{pstricks}
 \usepackage{epsfig}
 \usepackage{pst-grad} 
 \usepackage{pst-plot} 
 \usepackage[space]{grffile} 
 \usepackage{etoolbox} 
 \makeatletter 
 \patchcmd\Gread@eps{\@inputcheck#1 }{\@inputcheck"#1"\relax}{}{}
 \makeatother

\usepackage{amsmath}
\usepackage{amssymb}
\usepackage{amsthm}
\usepackage[latin1]{inputenc}
     
\usepackage{graphicx}

\usepackage[usenames,dvipsnames]{pstricks}
\usepackage{epsfig}
\usepackage{pst-grad} 
\usepackage{pst-plot} 

\usepackage{ifthen}
\usepackage{color}

\newcommand{\intav}[1]{\mathchoice {\mathop{\vrule width 6pt height 3 pt depth  -2.5pt
\kern -8pt \intop}\nolimits_{\kern -6pt#1}} {\mathop{\vrule width
5pt height 3  pt depth -2.6pt \kern -6pt \intop}\nolimits_{#1}}
{\mathop{\vrule width 5pt height 3 pt depth -2.6pt \kern -6pt
\intop}\nolimits_{#1}} {\mathop{\vrule width 5pt height 3 pt depth
-2.6pt \kern -6pt \intop}\nolimits_{#1}}}

\def\polhk#1{\setbox0=\hbox{#1}{\ooalign{\hidewidth\lower1.5ex\hbox{`}\hidewidth\crcr\unhbox0}}}

\newcommand{\dx}{\operatorname{d}}

\renewcommand{\dx}{\operatorname{d}}

\newtheorem{teo}{Theorem}

\newtheorem{Proposition}{Proposition}
\newtheorem{Remark}{Remark}
\newtheorem{Assumption}{A}

\begin{document}

\title{Geometric regularity for elliptic equations in double-divergence form}
\author{Raimundo Leit\~ao, Edgard A. Pimentel and Makson S. Santos}

\date{\today} 

\maketitle

\begin{abstract}

In this paper, we examine the regularity of the solutions to the double-divergence equation. We establish improved H\"older continuity as solutions approach their zero level-sets. In fact, we prove that $\alpha$-H\"older continuous coefficients lead to solutions of class $\mathcal{C}^{1^-}$, locally. Under the assumption of Sobolev differentiable coefficients, we establish regularity in the class $\mathcal{C}^{1,1^-}$. Our results unveil improved continuity along a nonphysical free boundary, where the weak formulation of the problem vanishes. We argue through a geometric set of techniques, implemented by approximation methods. Such methods connect our problem of interest with a target profile. An iteration procedure imports information from this limiting configuration to the solutions of the double-divergence equation.

\noindent 
\medskip

\noindent \textbf{Keywords}: Double-divergence equations; geometric regularity; improved regularity at zero level-sets.

\medskip 

\noindent \textbf{MSC(2010)}: 35B65; 35J15.
\end{abstract}

\vspace{.1in}

\section{Introduction}\label{introduction}
In the present paper we study the regularity theory for solutions to the double-divergence partial differential equation (PDE)
\begin{equation}\label{eq_main}
	\partial^2_{x_ix_j}\left(a^{ij}(x)u(x)\right)\,=\,0\;\;\;\;\;\mbox{in}\;\;\;\;\;B_1,
\end{equation}
where $(a^{ij})_{i,j=1}^d\in\mathcal{S}(d)$ is uniformly $(\lambda,\Lambda)$-elliptic. We produce new (sharp) regularity results for the solutions to \eqref{eq_main}. In particular, we are concerned with gains of regularity as solutions approach their zero level-sets. We argue through a genuinely geometric class of methods, inspired by the ideas introduced by L. Caffarelli in \cite{caffarelli}.

Introduced in \cite{littman2}, equations in the double-divergence form have been the object of important advances. See \cite{sjogren,bogachev1,pre1,pre2,fabesstroock}; see also the monograph \cite{bogachev2}. The interest in \eqref{eq_main} is due to its own mathematical merits, as well as to its varied set of applications.


The primary motivation for the study of \eqref{eq_main} is in the realm of stochastic analysis. In fact, \eqref{eq_main} is the Kolmogorov-Fokker-Planck equation associated with the stochastic process whose infinitesimal generator is given by
\[
	Lv(x)\,:=\,a^{ij}(x)\partial^2_{x_ix_j}v(x).
\]
Therefore, one can derive information on the stochastic process through the understanding of \eqref{eq_main}.

A further instance where double-divergence equations play a role is the fully nonlinear mean-field games theory. The model-problem here is
\begin{equation}\label{eq_mfg}
	\begin{cases}
	F(D^2V)\,=\,g(u)&\;\;\;\;\;\mbox{in}\;\;\;\;\;B_1\\
	\partial^2_{x_i x_j}\left(F^{ij}(D^2V)u(x)\right)\,=\,0&\;\;\;\;\;\mbox{in}\;\;\;\;\;B_1,
	\end{cases}
\end{equation}
where $F:\mathcal{S}(d)\to\mathbb{R}$ is a $(\lambda,\Lambda)$-elliptic operator, $F^{ij}(M)$ stands for the derivative of $F$ with respect to the entry $m_{i,j}$ of $M$ and $g:\mathbb{R}\to\mathbb{R}$ is a given function. In this case, the first equation in \eqref{eq_mfg} is a Hamilton-Jacobi, associated with an optimal control problem. Its solution $V$ accounts for the value function of the game. On the other hand, the population of players, whose density is denoted by $u$, solves a double-divergence (Fokker-Planck) equation. The mean-field coupling $g$ encodes the preferences of the players with respect to the density of the entire population. Therefore, the solution $u$ describes the equilibrium distribution of a population of rational players facing a scenario of strategic interaction. Through this framework, double-divergence equations are relevant in the modelling of several phenomena in the life and social sciences. As regards the mean-field games theory, we refer the reader to the monograph \cite{thebook}. 

A further application of equations in double-divergence form occurs in the theory of Hamiltonian stationary Lagrangian manifolds \cite{chen06}. Let $\Omega \subset \mathbb{R}^d$ be a domain and consider $u \in C^{\infty}(\Omega)$. The gradient graph of $u$ is the set 
\[
	\Gamma_u\,:=\, \left\lbrace \left(x, Du(x)\right), x \in \Omega\right\rbrace,
\]
whereas the volume of $\Gamma_u$ is given by
\[
	F_{\Omega}(u) = \int_{\Omega}\left( \det(I \,+\, (D^2u)^TD^2u)\right)^{1/2}\dx x.
\]
Given $\Omega \subset \mathbb{R}$, the study of critical points/minimizers for $F_{\Omega}(u)$ yields the compactly supported first variation
\begin{equation}\label{eq_geo}
	\int_{\Omega}\sqrt{\det g}g^{ij}\delta^{kl}u_{x_ix_k}\phi_{x_jx_l}\dx x = 0,
\end{equation}
for all $\phi \in C_c^{\infty}(\Omega)$, where
\[
	g\,:=\,I \,+\, (D^2u)^TD^2u
\] 
is the induced metric. It is easy to check that \eqref{eq_geo} is the weak (distributional)  formulation of
\[
	\partial^2_{x_jx_l}\left(\sqrt{\det g}g^{ij}\delta^{kl}u_{x_ix_k} \right) = 0\;\;\;\;\;\mbox{in}\;\;\;\;\;\Omega.
\]

Hence, given a domain, the minimizers of the volume of the gradient graph relate to the solutions of a PDE in the double-divergence form.

As mentioned before, the study of \eqref{eq_main} starts in \cite{littman2}. In that paper, the author considers weak solutions to the inequality
\begin{equation*}
	\partial^2_{x_ix_j}\left(a^{ij}(x)u(x)\right)\,\geq\,0\;\;\;\;\;\mbox{in}\;\;\;\;\;B_1,
\end{equation*}
and establishes a strong maximum principle. In \cite{pre1}, the author develops a potential theory associated with \eqref{eq_main}. This theory is shown to satisfy the same axioms as the potential theory for the elliptic operator
\[
	a^{ij}(\cdot)\partial^2_{x_ix_j}.
\]
Hence, the study of the former provides information on the latter. An improved maximum principle, as well as a preliminary approximation scheme for \eqref{eq_main} are the subject of \cite{pre2}.

It is only in \cite{sjogren} that the regularity for the solutions to \eqref{eq_main} is first investigated. In that paper, the author proves that solutions coincide with a continuous function, except in a set of measure zero. Together with its converse -- and under further conditions -- this is called the \emph{fundamental equivalence}. In addition, a result on the $\alpha$-H\"older continuity of the solutions is presented. Namely, solutions are proven to be locally $\alpha$-H\"older continuous provided the coefficients satisfy $a^{ij}\in\mathcal{C}^\alpha_{loc}(B_1)$.

In \cite{fabesstroock}, the authors examine properties of the Green's function associated with the operator driving \eqref{eq_main}. One of the results in that paper regards gains of integrability for the solutions. In fact, it is reported that locally integrable, non-negative, solutions are in $L^\frac{d}{d-1}_{loc}(B_1)$. 

A distinct approach to \eqref{eq_main} regards the study of the \textit{densities} of solutions. That is, their Radon-Nikodym derivatives with respect to the Lebesgue measure. In this realm, several developments have been produced (see \cite{bogachev2} and the references therein). For example it is widely known that, if $(a^{ij})_{i,j=1}^d$ is nondegenerate in $B_1$, every solution to \eqref{eq_main} has a density; see \cite{bogachev2}.

In \cite{bogachev3} the authors prove that $\det\left[(a^{ij})_{i,j=1}^d\right]u$ has a density in $L^\frac{d}{d-1}_{loc}(B_1)$, provided $u\geq 0$. If, in addition, $(a^{ij})_{i,j=1}^d$ is H\"older continuous and uniformly elliptic, $u$ is proven to have a density in $L^\frac{d}{d-1}_{loc}(B_1)$. Regularity in Sobolev spaces is also studied in \cite{bogachev3}. Under the assumptions that $(a^{ij})_{i,j=1}^d$ is in $W^{1,p}_{loc}(B_1)$ and $\det\left[(a^{ij})_{i,j=1}^d\right]$ is bounded away from zero, the authors prove that solutions have a density in $W^{1,p}_{loc}(B_1)$. It is worth noticing that \cite{bogachev3} addresses differential \emph{inequalities} of the form
\begin{equation*}\label{eq_boggen}
	\int_{B_1}a^{ij}(x)u(x)\phi_{x_ix_j}(x)dx\,\leq\,C\left\|\phi\right\|_{W^{1,\infty}(B_1)},
\end{equation*}
for some $C>0$. The corpus of results reported in \cite{bogachev3} refines important previous developments; see, for instance \cite{bogachev4,krylov1}.

In the recent paper \cite{bogachev1}, the authors consider densities of the solutions to \eqref{eq_main} and investigate their regularity in H\"older and Lebesgue spaces. In addition, they prove a Harnack inequality for non-negative solutions; see \cite[Corollary 3.6]{bogachev1}. Among other things, this result is relevant as it sets in the positive an open question raised in \cite{mamedov}. In fact, it is shown that densities are in $L^p_{loc}(B_1)$, for every $p\geq 1$, if $(a^{ij})_{i,j=1}^d\in VMO(B_1)$. Moreover, the authors examine the regularity of densities in H\"older spaces, provided the coefficients are in the same class. 

A remarkable feature of PDEs in the double-divergence form is the following: the regularity of $(a^{ij})_{i,j=1}^d$ acts as an upper bound for the regularity of the solutions. It means that gains of regularity are not (universally) available for the solutions, \emph{vis-a-vis} the data of the problem. To see this phenomenon in a (very) simple setting, we detail an example presented in \cite{bogachev1}. Set $d=1$ and consider the homogeneous problem
\begin{equation}\label{eq_example}
	\left(a(x)v(x)\right)_{xx}\,=\,0\;\;\;\;\;\mbox{in}\;\;\;\;\;]-1,1[.
\end{equation}
Take an arbitrary affine function $\ell:B_1\to\mathbb{R}$ and let $u(x):=\ell(x)/a(x)$. Notice that
\[
	\int_{B_1}a(x)\frac{\ell(x)}{a(x)}\phi_{xx}\dx x\,=\,0
\]
for every $\phi\in\mathcal{C}^2_0(]-1,1[)$. Therefore, $u$ is a solution to \eqref{eq_example}. It is clear that, if $a(x)$ is discontinuous, so will be $u$.

Although solutions lack gains of regularity in the entire domain, a natural question regards the conditions under which improvements on the H\"older continuity could be established. Let $S\subset B_1$ be a fixed subset of the domain and suppose that further, natural, conditions are placed on $\left(a^{ij}\right)_{i,j=1}^d\in\mathcal{C}^\beta_{loc}(B_1)$. An important information concerns the regularity of the solutions along $S$. Even more relevant in some settings is the regularity of the solutions as they \emph{approach} $S\subset B_1$. 

In this paper, we consider the zero level-set of the solutions to \eqref{eq_main}. That is, 
\[
	S_0[u]\,:=\,\left\lbrace x\,\in\,B_1\,:\,u(x)\,=\,0\right\rbrace.
\]	
We prove that, along $S_0$, solutions to \eqref{eq_main} are of class $\mathcal{C}^\alpha$ for every $\alpha\in(0,1)$, provided $\left(a^{ij}\right)_{i,j=1}^d$ is H\"older continuous and satisfies a proximity regime of the form
\[
	\left\|a^{ij}\,-\,a^{ij}(0)\right\|_{L^\infty(B_1)}\,\ll\,1/2.
\]
The precise statement of our first main result is the following:

\begin{teo}\label{thm_main}
Let $u\in\ L^1_{loc}(B_1)$ be a weak solution to \eqref{eq_main}. Suppose assumptions A\ref{assump_matrixa}-A\ref{assump_holder}, to be set forth in Section \ref{subsec_mainassump}, are in force. Let $x_0\in S_0(u)$. Then $u$ is of class $\mathcal{C}^{1-}$ at $x_0$ and there exists a constant $C>0$ such that
\[
 \sup_{B_r(x_0)}\left|u(x_0)\,-\,u(x)\right|\,\leq\,Cr^{\alpha^*},
\]
for every $\alpha^*\in(0,1)$.
\end{teo}

The contribution of Theorem \ref{thm_main} is to ensure gains of regularity for the solutions to \eqref{eq_main} \textit{as they approach the zero level-set}, though estimates in the whole domain are constrained by the regularity of the coefficients $a^{ij}$. From a heuristic viewpoint, whichever level of $\varepsilon$-H\"older continuity is available for the coefficients  -- with $0<\varepsilon\ll 1/2$ -- suffices to produce $\mathcal{C}^{1^-}$ regularity for the solutions along $S_0[u]$. 

The choice for $S_0$ is two-fold. Indeed, along this set, the weak formulation of \eqref{eq_main} vanishes. Hence, at least intuitively, the weak formulation of the problem fails to provide information on the original equation along $S_0[u]$. A remarkable feature of \eqref{eq_main} is related to this apparent lack of information across the zero level-set. As a matter of fact, the structure of the equation is capable of enforcing higher regularity of the solutions along the set where the weak formulation vanishes. 

A second instance of motivation for the choice of $S_0$ falls within the scope of the \emph{nonphysical free boundaries}. Introduced as a technology inspired by free boundary problems in the regularity theory of (nonlinear) partial differential equations, this class of methods has advanced the understanding of fine properties of solutions to a number of important examples. We refer the reader to \cite{edutei1}.

In addition to the study of \eqref{eq_main} in the presence of H\"older continuous coefficients, we also consider the case $\left(a^{ij}\right)_{i,j=1}^d\in W^{2,p}_{loc}(B_1)$, for $p>d$. In this setting, \eqref{eq_main} becomes
\begin{equation}\label{eq_newsob}
	\partial_{x_i}\left(a^{ij}(x)\partial_{x_j}u(x)\,+\,\partial_{x_j}a^{ij}(x)u(x)\right)\,=\,0\;\;\;\;\;\mbox{in}\;\;\;\;\;B_1.
\end{equation}
Here, two new layers of information are unveiled. First, it is known that solutions to \eqref{eq_newsob} are in $\mathcal{C}^{1,1-d/p}_{loc}(B_1)$ -- see \cite[Chapter 3, Theorem 15.1]{ladyural}. I.e., the gradient of the solutions exists in classical sense. Second, the weak formulation of the problem vanishes at a different subset of the domain, namely
\[
	S_1[u]\,:=\,\left\lbrace x\,\in\,B_1\,:\,u(x)\,=\,0\;\;\;\mbox{and}\;\;\;Du(x)\,=\,{\bf 0}\right\rbrace.
\]
Under the assumption $\left(a^{ij}\right)_{i,j=1}^d\in W^{2,p}_{loc}(B_1)$, and the appropriate proximity regime, we prove that solutions to \eqref{eq_main} are locally of class $\mathcal{C}^{1,1^-}$ along $S_1[u]$. This is the content of our second main result:

\begin{teo}[H\"older regularity of the gradient]\label{thm_main2}
Let $u\in\ L^1_{loc}(B_1)$ be a weak solution to \eqref{eq_main}. Suppose A\ref{assump_matrixa} and A\ref{assump_w2p}, to be introduced in Section \ref{subsec_mainassump}, hold true. Let $x_0\in S_1[u]$. Then $u$ is of class $\mathcal{C}^{1,1^-}$ at $x_0$ and there exists a constant $C>0$ such that
\[
	\sup_{x\in B_r(x_0)}\left|Du(x) - Du(x_0)\right| \leq Cr^{\alpha^*}
\]
for every $\alpha^*\in(0,1)$.
\end{teo}

The regularity of the coefficients in Sobolev spaces is pivotal in establishing Theorem \ref{thm_main2}. Here, Sobolev differentiable coefficients switch the regularity regime of \eqref{eq_main} allowing for an alternative weak formulation of the problem. 

We remark that our methods accommodate equations with explicit dependence on lower order terms. I.e., 
\[
	\partial^2_{x_ix_j}\left(a^{ij}(x)u(x)\right)\,-\,\partial_{x_i}\left(b^i(x)u(x)\right)\,+\,c(x)u(x)=\,0\;\;\;\;\;\mbox{in}\;\;\;\;\;B_1,
\]
provided the vector field $b:B_1\to\mathbb{R}^d$ and the function $c:B_1\to\mathbb{R}$ are well-prepared. See Remarks \ref{remarklot1} and \ref{remarklot2}.

\begin{figure}[h]
\center

\psscalebox{.7 .7} 
{
\begin{pspicture}(0,-3.4)(13.2,3.4)
\definecolor{colour0}{rgb}{0.91764706,0.91764706,0.91764706}
\psframe[linecolor=black, linewidth=0.04, fillstyle=solid,fillcolor=colour0, dimen=outer, framearc=0.1](13.2,3.4)(0.0,-3.4)
\pscircle[linecolor=black, linewidth=0.03, fillstyle=solid, dimen=outer](6.4,-2.2){0.8}
\psbezier[linecolor=black, linewidth=0.04](1.6,-1.8)(4.0,0.6)(3.6,2.6)(3.6,1.0)(3.6,-0.6)(5.6,0.6)(6.4,-2.2)
\psline[linecolor=black, linewidth=0.04, linestyle=dashed, dash=0.17638889cm 0.10583334cm](1.2,-2.2)(7.6,-2.2)
\psline[linecolor=black, linewidth=0.08](6.4,-2.2)(8.0,-2.2)
\rput[bl](7.6,-3.0){\LARGE{$S_0[u]$}}
\pscircle[linecolor=black, linewidth=0.1, fillstyle=solid, dimen=outer](9.6,0.2){2.4}
\psline[linecolor=black, linewidth=0.1](9.2,0.2)(12.0,0.2)
\psline[linecolor=black, linewidth=0.1](9.2,0.2)(8.4,2.2)
\rput[bl](8.0,-0.6){\large{Almost-Lipschitz} 
}
\rput[bl](8.4,-1.0){
\large{decay to zero}}
\rput[bl](2.0,0.6){\LARGE{$u(x)$}}
\psline[linecolor=black, linewidth=0.1, linestyle=dashed, dash=0.17638889cm 0.10583334cm](9.2,0.2)(7.2,0.2)
\rput[bl](0.8,-3.0){\LARGE{$B_1$}}
\end{pspicture}
}
\caption{Almost-Lipschitz decay to zero: although the graph of the solutions to \eqref{eq_main} admits cusps in the presence of merely H\"older continuous coefficients, they approach their zero level-sets with $\mathcal{C}^\alpha$-regularity, for every $\alpha\in(0,1)$. It means that solutions reach the nonphysical free boundary in an almost-Lipschitz manner.}

\end{figure}
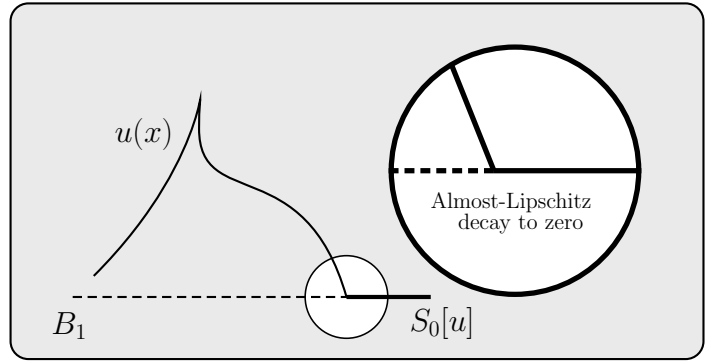

Our arguments are intrinsically geometric. We approximate weak solutions to \eqref{eq_main} by solutions to a homogeneous, fixed coefficients, equation of the form
\begin{equation}\label{eq_approximate}
	a^{ij}(0)\partial^2_{x_ix_j}v(x)\,=\,0\;\;\;\;\;\mbox{in}\;\;\;\;\;B_1.
\end{equation}

Among such solutions, we select $v$ such that $S_0[u]\subset S_0[v]$, and $S_1[v]\subset S_1[u]$, when appropriate. An approximation routine builds upon the regularity theory available for the solutions to \eqref{eq_approximate}. This is achieved through a geometric strategy, which produces a preliminary oscillation control. To turn this initial information into an oscillation control in every scale, an iterative method takes place. This line of reasoning is inspired by trail-blazing ideas first introduced in  \cite{caffarelli}. See also \cite{ccbook}.

The remainder of this paper is organized as follows: Section \ref{subsec_mainassump} details our main assumption whereas Section \ref{subsec_premres} collects a few elementary facts and notions, together with auxiliary results. In Section \ref{sec_proofmt} we put forward a zero level-set approximation lemma and present the proof of Theorem \ref{thm_main}. A finer approximation result appears in Section \ref{sec_proofthm2}, where we conclude the proof of Theorem \ref{thm_main2}.

\bigskip

\noindent {\bf Acknowledgements:} The authors are very grateful to an anonymous referee for her/his comments and suggestions, which led to a substantial improvement in the material in the paper. E. Pimentel is partially funded by CNPq-Brazil (Grants \#433623/2018-7 and \# 307500/2017-9) and FAPERJ (Grant \#E26/200.002/ 2018). M. Santos is supported by CAPES-Brazil. Part of this work was developed during the second and third authors' visit to the International Centre for Theoretical Physics (ICTP, Trieste); the authors are grateful for the Centre's support and hospitality. This study was financed in part by the Coordena\c{c}\~ao de Aperfei\c{c}oamento de Pessoal de N\'ivel Superior - Brazil (CAPES) - Finance Code 001"

\section{Preliminary material and main assumptions}\label{sec_mapm}

In this section we introduce the main elements used in our arguments throughout the paper. Firstly we discuss our assumptions on the structure of the problem. Then, we collect a few definitions and results. 

\subsection{Main assumptions}\label{subsec_mainassump}

In what follows, we detail the main hypotheses under which we work in the present paper. We start with an assumption on the uniform ellipticity of the coefficients matrix $(a^{ij})_{i,j=1}^d$.

\begin{Assumption}[Uniform ellipticity]\label{assump_matrixa}
We assume the symmetric matrix $(a^{ij}(x))_{i,\,j=1}^d$ satisfies a $(\lambda,\Lambda)$-ellipticity condition of the form 
\[
	\lambda Id\,\leq\,(a^{ij}(x))_{i,\,j=1}^d\,\leq\,\Lambda Id,
\]
for some $0<\lambda\leq\Lambda$, uniformly in $x\in B_1$.
\end{Assumption}

The next assumption concerns the regularity requirements on the coefficients to ensure H\"older continuity of the solutions to \eqref{eq_main}. This fact is central in the proof of Theorem \ref{thm_main}.

\begin{Assumption}[$\alpha$-H\"older continuity]\label{assump_holder} The map $(a^{ij}(x))_{i,\,j=1}^d:B_1\to\mathcal{S}(d)$ is locally uniformly $\alpha$-H\"older continuous. That is, we have 
\[
	a^{ij}\,\in\,\mathcal{C}_{loc}^\alpha(B_1)
\]
for every $1\leq i\leq d$ and $1\leq j\leq d$.
\end{Assumption}



We conclude this section with a further set of conditions on the coefficients $a^{ij}$. Such assumption unlocks the study of the gradient-regularity for the solutions to \eqref{eq_main}, along $S_1[u]$.

\begin{Assumption}[Sobolev differentiability of the coefficients]\label{assump_w2p} Let $p>d$. The map 
\[
	\left(a^{ij}\right)_{i,\,j=1}^d:B_1\to\mathcal{S}(d)
\]
is in $W^{2,p}_{loc}(B_1)$. That is, we have 
\[
	a^{ij}\,\in\,W_{loc}^{2,p}(B_1),
\]
for every $1\leq i\leq d$ and $1\leq j\leq d$.
\end{Assumption}

In the next section we gather elementary notions and basic facts used further in the paper.

\subsection{Preliminary notions and results}\label{subsec_premres}

We start with a result first proven in \cite{sjogren}. It concerns the existence of a continuous version to the weak solutions to \eqref{eq_main}.

\begin{Proposition}[Continuous version of weak solutions]\label{prop_contversion}
Let $u\in L^1_{loc}(B_1)$ be a weak solution to \eqref{eq_main}. Then, there exists a null set $\Omega\subset B_1$ and $v\in\mathcal{C}(B_1)$ such that
\[
	u\,\equiv\,v\;\;\;\;\;\;\;\;\;\;\mbox{in}\;\;\;\;\;\;\;\;\;\;B_1\,\setminus\,\Omega.
\] 
\end{Proposition}
\begin{proof}
For the proof of the proposition, we refer the reader to \cite[Lemma 1]{sjogren} ; see also \cite{sjogren75}.
\end{proof}
\begin{Remark}
Hereafter, we suppose that every locally integrable function solving \eqref{eq_main} in the weak sense is continuous.
\end{Remark}

Before proceeding we recall the fundamental solution of the operator 
\[
	a^{ij}(y)\partial^2_{x_ix_j};
\]
such function will be denoted by $H(x,y)$. In the case $d>2$, $H$ is defined as
\begin{equation}\label{eq_fundamentalsol}
	H(x,y)\,:\,=\frac{\left[a_{ij}(y)\left(x_i\,-\,y_i\right)\left(x_j\,-\,y_j\right)\right]^\frac{2-d}{2}}{\left(d\,-\,2\right)\alpha(d)\sqrt{\det{[(a^{ij})_{i,j=1}^d]}}},
\end{equation}
where $(a_{ij})_{i,j=1}^d:=[(a^{ij})_{i,j=1}^d]^{-1}$ and $\alpha(d)$ stands for the volume of the unit ball in dimension $d$.

A fundamental result in the context of this paper regards initial levels of compactness for the solutions to \eqref{eq_main}. This is the subject of the next proposition, which we recall here for the sake of completeness. 

\begin{Proposition}[Compactness of the solutions]\label{prop_compactnessol}
Let $u\in L^1_{loc}(B_1)$ be a weak solution to \eqref{eq_main}. Suppose A\ref{assump_matrixa}- A\ref{assump_holder} are in force. Then, $u\in\mathcal{C}^\alpha_{loc}(B_1)$ and there exists a constant $C>0$ such that
\begin{equation}\label{eq_estimate}
	\left\|u\right\|_{\mathcal{C}^\alpha(B_{1/2})}\,\leq\,C,
\end{equation}
with $C=C\left(d,\lambda,\Lambda,\left\|a^{ij}\right\|_{\mathcal{C}^\alpha(B_1)},\left\|u\right\|_{L^\infty(B_1)}\right)$.
\end{Proposition}
\begin{proof}
The inclusion $u\in\mathcal{C}^\alpha_{loc}(B_1)$ is a well-known result; see, for instance \cite[Theorem 2]{sjogren}. As for the estimate in \eqref{eq_estimate} it follows from considerations on the oscillation of the fundamental solution $H$, defined in \eqref{eq_fundamentalsol}, and its derivatives; see the proof of \cite[Theorem 2]{sjogren}.
\end{proof}

We proceed with a proposition on the sequential stability of the solutions to \eqref{eq_main}. It will be used further to establish two approximation lemmas.

\begin{Proposition}[Sequential stability of weak solutions]\label{prop_seqstab}
Suppose that  
\[
	\left([a^{ij}_n]_{i,j=1}^d\right)_{n\in\mathbb{N}}\subset\mathcal{C}^\alpha_{loc}(B_1;\mathcal{S}(d))
\] 
is a sequence of matrices such that 
\[
	\left\|a^{ij}_n\,-\,a^{ij}(x_0)\right\|_{L^\infty(B_1)}\,\to\,0
\]
as $n\to \infty$. Suppose further that $(f_n)_{n\in\mathbb{N}}\subset L^p(B_1)$ is so that 
\[
	\left\|f_n\right\|_{L^p(B_1)}\,\to\,0
\]
as $n\to\infty$.  Let $(u_n)_{n\in\mathbb{N}}\subset L^1_{loc}(B_1)$ satisfy
\[
	\partial^2_{x_ix_j}\left(a^{ij}_n(x)u_n(x)\right)\,=\,f_n\;\;\;\;\;\mbox{in}\;\;\;\;\;B_1.
\]
If there exists $u_\infty\in\mathcal{C}(B_1)$ such that 
\[
	\left\|u_n\,-\,u_\infty\right\|_{L^\infty(B_1)}\,\to\,0
\]
as $n\to\infty$, then $u_\infty$ satisfies
\[
	\int_{B_1}a^{ij}(x_0)u_{\infty}(x)\phi_{x_ix_j}(x)dx\,=\,0
\]
for every $\phi\in \mathcal{C}^2_c(B_1)$.
\end{Proposition}
\begin{proof}
First, notice that we have $a^{ij}_n(x_0)\to a^{ij}(x_0)$ as $n\to\infty$. Now, for every $\phi\in\mathcal{C}^2_c(B_1)$ we have

\begin{align*}
\left|\int_{B_1}\phi_{x_ix_j}a^{ij}(x_0)u_\infty(x)\dx x\right| &\leq\int_{B_1}|\phi_{x_ix_j}|\left|a^{ij}(x_0)-a^{ij}_n(x)\right||u_\infty(x)|\dx x\\
&\quad +\int_{B_1}|\phi_{x_ix_j}||a^{ij}_n(x)|\left|u_n(x)-u_\infty(x)\right|\dx x \\
&\quad + \int_{B_1}|\phi||f_n|\dx x.
\end{align*}
Notice that the right-hand side of this inequality converges to zero as $n\to\infty$. Therefore, 
\[
	\int_{B_1}\phi_{x_ix_j}a^{ij}(x_0)u_\infty(x)\dx x\,=\,0.
\]
This concludes the proof.
\end{proof}

In addition to the sequential stability, our arguments require an initial degree of compactness for the solutions to \eqref{eq_main}. When it comes to the proof of Theorem \ref{thm_main}, uniform compactness comes from Proposition \ref{prop_compactnessol}. In the case of Theorem \ref{thm_main2}, we turn to a well-known result on the regularity of the (weak) solutions to equations in the divergence form. We start with an observation.

In case A\ref{assump_w2p} is in force, we claim that \eqref{eq_main} can be written as
\begin{equation}\label{eq_divlady}
	\partial_{x_i}\left(a^{ij}(x)\partial_{x_j}u(x)\,+\,\partial_{x_j}a^{ij}(x)u(x)\right)\,=\,0\;\;\;\;\;\mbox{in}\;\;\;\;\;B_1.
\end{equation}
Indeed, if $a^{ij}$ is weakly differentiable, we have
\begin{align*}
	\int_{B_1}a^{ij}u\partial_{x_ix_j}\phi \dx x\,=\,-\int_{B_1}\left(a^{ij}\partial_{x_j}u\,+\,\partial_{x_j}a^{ij}u\right)\partial_{x_i}\phi\dx x,
\end{align*}
for every $\phi\in\mathcal{C}^2_c(B_1)$. Hence, under A\ref{assump_w2p}, the homogeneous version of \eqref{eq_main} is equivalent to \eqref{eq_divlady}. Now we are in position to state the following:

\begin{Proposition}\label{prop_lady}
Let $v\in W^{1,p}(B_1)$ be a weak solution to \eqref{eq_divlady}. Suppose A\ref{assump_matrixa} and A\ref{assump_w2p} are in force. Then, $v\in \mathcal{C}^{1,\alpha}_{loc}(B_1)$, where 
\[
	\alpha\,:=\,\frac{p\,-\,d}{p}.
\]
Moreover, there exists a universal constant $C>0$ such that 
\[
	\left\|v\right\|_{\mathcal{C}^{1,\alpha}(B_{1/2})}\,\leq\,C\left\|v\right\|_{L^\infty(B_1)}.
\]
\end{Proposition}

For the proof of Proposition \ref{prop_lady}, we refer the reader to \cite[Chapter 3, Theorem 15.1]{ladyural}. The former proposition is paramount in establishing Theorem \ref{thm_main2}. Apart from compactness, it produces gradient-continuity for the solutions to \eqref{eq_divlady}. This information plays a critical role in the treatment of fine regularity properties of the solutions to the homogeneous version of \eqref{eq_main} along $S_1[u]$. In particular, it unlocks a first zero level-set approximation result.

We conclude this section with a comment on the scaling properties of \eqref{eq_main}. Indeed, we consider weak solutions satisfying $\left\|u\right\|_{L^\infty(B_1)}\leq 1$. Let $\overline{u}\in\mathcal{C}(B_1)$ be defined as follows:
\[
	\overline{u}(x)\,:=\,\frac{u(x)}{\max\left\lbrace 1,\,\left\|u\right\|_{L^\infty(B_1)}\right\rbrace},
\]	
where $u$ is a weak solution to \eqref{eq_main}. It is clear that $\overline{u}$ is a weak solution to
\[
	\partial^2_{x_ix_j}\left(a^{ij}(x)\overline{u}(x)\right)\,=\, 0\;\;\;\;\;\mbox{in}\;\;\;\;\;B_1.
\]
Notice that $\left\|\overline{u}\right\|_{L^\infty(B_1)}\leq 1$. Then, hereinafter we consider, without loss of generality, normalized solutions to \eqref{eq_main}. In the sequel, we set forth the proof of Theorem \ref{thm_main}.

\section{Improved regularity of the solutions}\label{sec_proofmt}

In this section we detail the proof of Theorem \ref{thm_main}. As mentioned before, we reason through an approximation/geometric method. At the core of our argument lies a zero level-set Approximation Lemma. It reads as follows:

\begin{Proposition}[Zero level-set Approximation Lemma]\label{prop_approx}
Let $u\in L^1_{loc}(B_1)$ be a weak solution to \eqref{eq_main}, $x_0 \in S_0[u]\cap B_{9/10}$ and suppose A\ref{assump_matrixa}-A\ref{assump_holder} are in force. Given $\delta>0$, there exists $\varepsilon=\varepsilon(\delta)>0$ such that, if
\[
	\sup_{x\in B_1}\,\left|a^{ij}(x)\,-\,a^{ij}(0)\right|\,<\, \varepsilon,
\]
there exists $h\in \mathcal{C}^{1,1}(B_{9/10})$ satisfying
\[
	\left\|u\,-\,h\right\|_{L^\infty(B_{9/10})}\,<\,\delta
\]
with
\[
	h(x_0)\,=\,0.
\]
\end{Proposition}
\begin{proof}
The proof follows from a contradiction argument. We start by supposing that the statement of the proposition is false. Therefore, there exist $\delta_0>0$ and sequences $\left([a^{ij}_n]_{i,j=1}^d\right)_{n\in\mathbb{N}}$ and  $(u_n)_{n\in\mathbb{N}}\subset L^\infty(B_1)$ such that
\[
	\sup_{x\in B_1}\left|a^{ij}_n(x)\,-\,a^{ij}_n(0)\right|\,\sim\,\frac{1}{n},
\]
\[
x_0 \in S_0[u_n]\cap B_{9/10}
\]
and 
\[
	\partial^2_{x_i x_j}\left(a^{ij}_n(x)u_n(x)\right)\,=\,0\;\;\;\;\;\mbox{in}\;\;\;\;\;B_1,
\]
but
\[
	\left|u_n(x)\,-\,h(x)\right|\,>\,\delta_0 \;\;\;\;\;\;\;\;\;\;\mbox{or}\;\;\;\;\;\;\;\;\;\;h(x_0)\,\neq\,0
\]
for every $h\in\mathcal{C}^{1,1}(B_{9/10})$ and every $n\in\mathbb{N}$. 

Notice that $(u_n)_{n\in\mathbb{N}}$ is uniformly bounded in $\mathcal{C}^\alpha(B_1)$. Therefore, there exists $u_\infty$ such that 
\[
	\left\|u_n\,-\,u_\infty\right\|_{\mathcal{C}^\beta(B_1)}\,\to \,0,
\]
for every $0<\beta<\alpha$, through a subsequence, if necessary. On the other hand, we have that $a^{ij}_n(0) \to \overline{a}^{ij}(0)$ as $n \to \infty$; hence
\[
|a^{ij}_n(x) - \overline{a}^{ij}(0)| \leq |a^{ij}_n(x) - a^{ij}_n(0)| + |a^{ij}_n(0) - \overline{a}^{ij}(0)|.
\]
Therefore
\[
	\|a^{ij}_n - \overline{a}^{ij}(0)\|_{L^{\infty}(B_1)}\,\to\,0,
\]
as $n \to \infty$.
Hence, the sequential stability of weak solutions (Proposition \ref{prop_seqstab}) leads to
\[
	\partial^2_{x_ix_j}\left(\overline{a}^{ij}(0)u_\infty(x)\right)\,=\,0\;\;\;\;\;\mbox{in}\;\;\;\;\;B_{9/10}.
\]
The regularity theory for constant-coefficients equations implies that $u_\infty\in\mathcal{C}^{1,1}(B_{9/10})$ and, moreover, $u_\infty(x_0)=0$. Finally, there exists $N\in\mathbb{N}$ such that
\[
	\left|u_n(x)\,-\,u_\infty(x)\right|\,<\,\delta_0,
\] 
provided $n>N$. By taking $h\equiv u_\infty$, we produce a contradiction and conclude the proof.
\end{proof}

\begin{Remark}\label{remarkp5}
The proof of Proposition \ref{prop_approx} shows that the approximating function $h$ solves the problem
\begin{equation}\label{eq_hnorm}
	\begin{cases}
		\partial^2_{x_ix_j}\left(\overline{a}^{ij}(0)h(x)\right)\,=\,0&\;\;\;\;\;\mbox{in}\;\;\;\;\;B_{9/10}\\
		h\,=\,h_0&\;\;\;\;\;\mbox{on}\;\;\;\;\;\partial B_{9/10},
	\end{cases}
\end{equation}
where
\[
	\left\|h_0\right\|_{L^\infty(\partial B_{9/10})}\,\leq\,\delta\,+\,\left\|u\right\|_{L^\infty(B_1)}.
\]
Therefore, it follows from standard results in elliptic regularity theory that
\[
	\left\|h\right\|_{\mathcal{C}^{1,1}(B_{9/10})}\,\leq\,C\left(1\,+\,\left\|u\right\|_{L^\infty(B_{1})}\right),
\]
where $C>0$ depends on the dimension $d$, the ellipticity constants $\lambda$ and $\Lambda$ and $\overline{a}^{ij}(0)$. We notice the constant $C$ does not depend on $u$.
\end{Remark}
\begin{Remark}\label{remarkepsilon}
A priori, the parameter $\varepsilon>0$ depends only on $\delta>0$. We notice however that (a universal) choice of $\delta$, made further in the paper, implies that $\varepsilon$ will depend on the exponent $\alpha$, the dimension $d$, $\lambda$, $\Lambda$ and $\left\|u\right\|_{L^\infty(B_1)}$. Therefore, we have
\[
	\varepsilon\,=\,\varepsilon\left(\alpha,\,d,\,\lambda,\,\Lambda,\,\left\|u\right\|_{L^\infty(B_1)}\right).
\]
\end{Remark}

Next, we control the oscillation of the solutions to \eqref{eq_main} within a ball of radius $0<\rho\ll1/2$, to be determined further.

\begin{Proposition}\label{prop_step1}
Let $u\in L^1(B_1)$ be a weak solution to \eqref{eq_main}. Suppose A\ref{assump_matrixa}- A\ref{assump_holder} are in force. Then, for every $\alpha\in(0,1)$, there exists 
$\varepsilon>0$ such that, if $x_0 \in S_0[u]\cap B_{9/10}$ and
\[
    \sup_{x\in B_1} \left|a^{ij}(x)\,-\,a^{ij}(0)\right|\,<\,\varepsilon,
\]
we can find $0<\rho\ll 1/2$ for which
\[
     \sup_{B_\rho(x_0)}\left|u(x)\right|\,\leq\,\rho^\alpha.
\]
\end{Proposition}
\begin{proof}
We start by taking a function $h\in\mathcal{C}^{1,1}_{loc}(B_{9/10})$ satisfying
\[
     \left\|u\,-\,h\right\|_{L^\infty(B_{9/10})}\,<\,\delta,
\]
with
\[
     h(x_0)\,=\,0.
\]

The existence of such a function is guaranteed by Proposition \ref{prop_approx}. We have
\[
     \sup_{x\in B_{\rho}(x_0)}\left|h(x)\,-\,h(x_0)\right|\,\leq\, C\rho,
\]
for some constant $C>0$;  see Remark \ref{remarkp5}. Therefore,
\begin{align}
\nonumber \sup_{x\in B_{\rho}(x_0)}\left|u(x)\,-\,h(x_0)\right|\,&\leq\,\sup_{x\in B_{\rho}(x_0)}\left|u(x)\,-\,h(x)\right|\,+\,\sup_{x\in B_{\rho}(x_0)}\left|h(x)\,-\,h(x_0)\right|
\\\label{eq_universal1}&\leq\,\delta\,+\,C\rho.
\end{align}

In the sequel, we make universal choices for $\rho$ and $\delta$; in fact, for a given $\alpha\in(0,1)$, we set
\begin{equation}\label{eq_universal2}
     \rho\,:=\,\left(\frac{1}{2C}\right)^\frac{1}{1-\alpha}\;\;\;\;\;\mbox{and}\;\;\;\;\;\delta\,:=\,\frac{\rho^\alpha}{2}.
\end{equation}
Finally, we combine \eqref{eq_universal1} with \eqref{eq_universal2} to obtain
\[
     \sup_{B_\rho(x_0)}\left|u(x)\right|\,\leq\,\rho^\alpha
\]
and conclude the proof.
\end{proof}

\begin{Proposition}\label{prop_finalstep}
Let $u\in L^1_{loc}(B_1)$ be a weak solution to \eqref{eq_main}. Suppose assumptions A\ref{assump_matrixa}-A\ref{assump_holder} are in force. Then, there exists 
$\varepsilon>0$ so that, if $x_0 \in S_0[u]\cap B_{9/10}$ and
\[
     \sup_{x\in B_1}\left|a^{ij}(x)\,-\,a^{ij}(0)\right|\,<\,\varepsilon,
\]
we can find $0<\rho\ll 1/2$ for which
\[
     \sup_{B_{\rho^n}(x_0)}\left|u(x)\right|\,\leq\,\rho^{n\alpha},
\]
for every $n\in\mathbb{N}$.
\end{Proposition}
\begin{proof}
We resort to an induction argument. First, we make the same choices as in \eqref{eq_universal2}; this (universally) determines the parameter $\varepsilon$. The first step of induction -- the case $n=1$ -- follows from Proposition \ref{prop_step1}. The induction hypothesis refers to the case $n=k$; i.e.,

\[
     \sup_{B_{\rho^k}(x_0)}\left|u(x)\right|\,\leq\,\rho^{k\alpha},
\]
for some $k\in\mathbb{N}$. 

In the sequel we address the case $n=k+1$. To that end, we introduce an auxiliary function $v_k:B_1\to \mathbb{R}$, defined as
\[
     v_k(x)\,:=\,\frac{u(x_0 + \rho^k x)}{\rho^{k\alpha}}.
\]

We observe that $v_k(0) = 0$. In addition $v_k$ solves 
\begin{equation}\label{eq_approxk}
    \partial^2_{x_ix_j}\left(a^{ij}_k(x)v_k(x)\right)\,=\,  0   \;\;\;\;\;\mbox{in}\;\;\;\;\;B_1,
\end{equation}
where
\[
	a^{ij}_k(x)\,:=\,a^{ij}(x_0 + \rho^k x).
\]
Now, notice that
\[
	\left|a^{ij}_k(x)\,-\,a^{ij}(0)\right|\,=\,\left|a^{ij}(x_0 + \rho^kx) \,- \,a^{ij}(0)\right|\, \leq\, \varepsilon.
\]
Finally, the matrix $(a^{ij}_k)_{i,j=1}^d$ inherits the H\"older continuity and the $(\lambda,\Lambda)$-ellipticity of $(a^{ij})_{i,j=1}^d$. Therefore, \eqref{eq_approxk} falls within the scope of Proposition \ref{prop_step1}. 
Hence,
\[
     \sup_{B_{\rho^k}}\left|v_k(x)\right|\,\leq\,\rho^{\alpha};
\]
by rescaling back to the unitary setting, we get
\[
     \sup_{B_{\rho^{k+1}}(x_0)}\left|u(x)\right|\,\leq\,\rho^{(k+1)\alpha}
\]
and complete the proof.
\end{proof}


\begin{proof}[Proof of Theorem \ref{thm_main}]
Let $0<r\ll1/2$ be fixed and take $x_0\in S_0[u]$. We must verify that 
\[
	\sup_{B_r(x_0)}\left|u(x)\,-\,u(x_0)\right|\,\leq\,Cr^\alpha,
\]
where $C>0$ is universal. Fix $n\in\mathbb{N}$ such that $\rho^{n+1}\leq r\leq \rho^n$. Observe that
\begin{align*}
	\sup_{B_r(x_0)}\left|u(x)\,-\,u(x_0)\right|\,\leq\,\sup_{B_{\rho^n}(x_0)}\left|u(x)\,-\,u(x_0)\right|\,\leq\,\rho^{-\alpha}\rho^{(n+1)\alpha}\,\leq\,Cr^\alpha.
\end{align*}
\end{proof}

We conclude this section with a remark on double divergence equations with explicit dependence on lower order terms.

\begin{Remark}\label{remarklot1}
To extend our result to model-problems of the form
\[
	\partial^2_{x_ix_j}\left(a^{ij}(x)u(x)\right)\,+\,\partial_{x_i}\left(b^i(x)u(x)\right)\,+\,c(x)u(x)=\, 0\;\;\;\;\;\mbox{in}\;\;\;\;\;B_1,
\]
it suffices to impose two conditions on $b:B_1\to\mathbb{R}^d$ and $c:B_1\to\mathbb{R}$. Indeed, these maps must be H\"older continuous; such a requirement unlocks the uniform compactness of the solutions. Secondly, a proximity regime must be in force; that is, there must be $\overline{b}\in\mathbb{R}^d$ and $\overline{c}\in\mathbb{R}$ so that
\[
	\left\|b^i\,-\,\overline{b^i}\right\|_{L^\infty(B_1)}\,+\,\left\|c\,-\,\overline{c}\right\|_{L^\infty(B_1)}\,\ll\,1/2.
\]
\end{Remark}

In what follows we focus on the proof of Theorem \ref{thm_main2}.

\section{H\"older continuity of the gradient}\label{sec_proofthm2}

This section sets forth the proof of Theorem \ref{thm_main2}. As before, the main ingredient is a First level-set Approximation Lemma. 

\begin{Proposition}[First Level-set Approximation Lemma]\label{prop_approx2}
Let $u \in L^1_{loc}(B_1)$ be a weak solution to \eqref{eq_main} and suppose A\ref{assump_matrixa} and A\ref{assump_w2p} are in force. Given $\delta > 0$, there exists $\varepsilon > 0$ such that, if $x_0 \in S_1[u]\cap B_{9/10}$ and
\[
	\sup_{x\in B_1}\left|a^{ij}(x) - a^{ij}(x_0)\right| \,< \,\varepsilon,
\]
there exists $h \in C^{1,1}(B_{9/10})$ satisfying 
\[
	\left\|u\, -\, h\right\|_{\mathcal{C}^{1,\beta}(B_{9/10})} < \delta,
\]
for some $\beta\in(0,1)$, with
\[
h(x_0) = 0\;\;\;\;\;\;\;\;\;\;\mbox{and}\;\;\;\;\;\;\;\;\;\; Dh(x_0) = {\bf 0}.
\]
\end{Proposition}

\begin{proof}
We argue by contradiction. Suppose the statement of the proposition is false, in this case there exists $\delta_0 > 0$ and sequences $\left([a^{ij}_n]_{i,j=1}^d\right)_{n\in\mathbb{N}}$, $(u_n)_{n\in\mathbb{N}}$ such that
\[
\left\|a_n^{ij}(x) - a_n^{ij}(x_0)\right\|_{L^\infty(B_1)} \,\sim\, \dfrac{1}{n},
\]
\[
x_0 \in S_1[u_n]\cap B_{9/10},
\]
and
\[
	\partial^2_{x_ix_j}\left(a_n^{ij}(x)u_n(x)\right) \,= \,0 \;\;\;\;\;\mbox{in}\;\;\;\;\; B_1,
\] 
with
\[
	|u_n(x) - h(x)|\,> \,\delta_0,
\]
 and either $h(x_0)\neq 0$ or $Dh(x_0)\neq {\bf 0}$, for every $h \in C^{1,1}(B_{9/10})$ and $n\in\mathbb{N}$. By Proposition \ref{prop_lady} we have that $(u_n)_{n\in \mathbb{N}}$ is uniformly bounded in $C^{1,\alpha}(B_1)$. Then, through a subsequence, if necessary, there exists a function $u_\infty$ such that
\[
	\left\|u_n\,-\,u_\infty\right\|_{C^{1,\gamma}(B_1)} \rightarrow 0,
\]
for every $0 \,<\,\gamma\,<\,\beta$. In particular 
\[
	u_n(x_0) \;\rightarrow \;u_\infty(x_0)\;\;\;\;\;\mbox{and}\;\;\;\;\;Du_n(x_0) \;\rightarrow \;Du_\infty(x_0).
\]
Then $u_\infty(x_0) = 0$ and $Du_\infty(x_0) = {\bf 0}$. Furthermore, $a_n^{ij}(x_0) \rightarrow \overline{a}^{ij}(x_0)$ as $n \rightarrow \infty$, hence, as before, $a_n^{ij}(x) \rightarrow \overline{a}^{ij}(x_0)$ as $n \rightarrow \infty$. 

Here, we evoke once again the sequential stability of the weak solutions, Proposition \ref{prop_seqstab}, to conclude that $u_\infty$ solves
\[
	\partial^2_{x_ix_j}\left({\overline a}^{ij}(x_0)u_\infty(x)\right) = 0 \;\;\;\;\;\mbox{in} \;\;\;\;\; B_{9/10}
\]
The regularity theory for constant coefficients implies that $u_{\infty} \in C^{1,1}(B_{9/10})$. By taking $h \equiv u_\infty$, we produce a contradiction and establish the result.
\end{proof}

\begin{Remark}\label{rem_prop8}
As in Remark \ref{remarkp5}, we notice that the norm of $h$ in $\mathcal{C}^2$ depends on the solution $u$ only through its $L^\infty$-norm.
\end{Remark}

\begin{Proposition}\label{prop_step1;2}
Let $u \in L^1_{loc}(B_1)$ be a weak solution to \eqref{eq_main} and suppose A\ref{assump_matrixa} and A\ref{assump_w2p} are in force. Then, for every $\alpha \in (0,1)$, there exists $\varepsilon >0$ such that, if $x_0 \in S_1[u]\cap B_{9/10}$ and
\[
	\sup_{x\in B_1}|a^{ij}(x) - a^{ij}(x_0)| < \varepsilon
\]
we can find $0 < \rho << 1/2$ such that
\[
	\sup_{B_\rho(x_0)}\left|Du(x)\,-\,Du(x_0)\right|\, \leq\, \rho^\alpha.
\]
\end{Proposition}

\begin{proof}
By Proposition \ref{prop_approx2}, there exists $h \in C^{1,1}(B_1)$ such that
\[
	\|u\,-\,h\|_{\mathcal{C}^{1,\beta}(B_{9/10})} \,< \,\delta  
\]
with $x_0\in S_1[u]\cap B_{9/10}$. 
We have
\begin{align*}
\sup_{B_\rho(x_0)}\left|Du(x)\,-\,Du(x_0)\right|&\leq \sup_{B_\rho(x_0)}\left|Du(x)-Dh(x)\right|+\sup_{B_\rho(x_0)}\left|Dh(x)-Dh(x_0)\right|\\&\quad+\sup_{B_\rho(x_0)}\left|Dh(x_0)-Du(x_0)\right|\\&\leq \delta\,+\,C\rho
\end{align*}

Now, by choosing 
\[
\rho := \left(\dfrac{1}{2C}\right)^{\frac{1}{1-\alpha}} \;\;\;\;\;\mbox{and} \;\;\;\;\; \delta := \dfrac{\rho^{\alpha}}{2},
\]
we obtain
\[
	\sup_{B_\rho(x_0)}\left|Du(x)\,-\,Du(x_0)\right|\, \leq\, \rho^{\alpha}
\]
and finish the proof.
\end{proof}

\begin{Proposition}
Let $u \in L^1_{loc}(B_1)$ be a weak solution to \eqref{eq_main} and suppose A\ref{assump_matrixa} and A\ref{assump_w2p} are in force. Then, there exists $\varepsilon > 0$ such that, if $x_0 \in S_1[u]\cap B_{9/10}$ and
\[
	\sup_{x\in B_1}\left|a^{ij}(x) \,-\, a^{ij}(x_0)\right|\, <\, \varepsilon,
\]
we can find $0 < \rho << 1/2$ for which
\[
	\sup_{B_{\rho^n}(x_0)}\left|Du(x)\,-\,Du(x_0)\right|\, \leq \,\rho^{n\alpha},
\]
for every $n \in \mathbb{N}$ and every $\alpha\in(0,1)$.
\end{Proposition}

\begin{proof}
We shall verify the proposition by induction. Notice that Proposition \ref{prop_step1;2} amounts to the first step in the induction argument. Suppose we have verified the statement for $n=k$. It remains to verify it in the case $n=k+1$. Define the function 
\[
	v_k(x) \,:=\, \dfrac{u(x_0\,+\,\rho^kx)}{\rho^{k(1+\alpha)}}.
\]
We start by noting that $0\in S_1[v_k]$. Besides, $v_k$ solves
\begin{equation}\label{eq_vkgrad}
	\partial^2_{x_ix_j}\left(a_k^{ij}(x)v_k(x) \right) \,=\, 0\;\;\;\;\;\mbox{in}\;\;\;\;\;B_1,
\end{equation}
where
\[
	a^{ij}_k(x)\,:=\,a^{ij}(x_0\,+\,\rho^kx).
\]
It is clear that, 
\[
\int_{B_1}|a^{ij}(x_0+\rho^k x)|^p dx = \dfrac{1}{\rho^{dk}}\int_{B_{\rho^k}(x_0)}|a^{ij}(y)|^p dy < C,
\]
where the inequality follows from  A\ref{assump_matrixa}. Also,
\[
\int_{B_1}|D(a^{ij}(x_0+\rho^k x))|^p dx = \rho^{k(p-d)}\int_{B_{\rho^k}(x_0)}|Da^{ij}(y)|^p dy < C,
\] 
since $p>d$, by hypothesis. Similarly
\[
\int_{B_1}|D^2(a^{ij}(x_0+\rho^k x))|^p dx = \rho^{k(2p-d)}\int_{B_{\rho^k}(x_0)}|D^2a^{ij}(y)|^p dy < C.
\]
Hence, \eqref{eq_vkgrad} falls within the scope of Proposition \ref{prop_step1;2}. Therefore 
\[
\sup_{B_\rho}\left|Dv_k(x)\,-\,Dv_k(0)\right|\, \leq\, \rho^{\alpha}.
\]
Re-scaling back to the unit ball, the former inequality implies
\[
\sup_{B_{\rho^{k+1}}(x_0)}\left|Du(x)\,-\,Du(x_0)\right|\, \leq\, \rho^{(k+1)\alpha}.
\]
This completes the proof.
\end{proof}

\begin{proof}[Proof of Theorem \ref{thm_main2}]
The proof follows the general lines of proof of Theorem \ref{thm_main} and will be omitted.
\end{proof}

\begin{Remark}\label{remarklot2}
As with in the previous case is possible to extend this result to model-problems of the form
\[
	\partial^2_{x_ix_j}\left(a^{ij}(x)u(x)\right)\,+\,\partial_{x_i}\left(b^i(x)u(x)\right)\,+\,c(x)u(x)=\,f(x)\;\;\;\;\;\mbox{in}\;\;\;\;\;B_1.
\]
As before, it suffices to impose two conditions on $b:B_1\to\mathbb{R}^d$ and $c:B_1\to\mathbb{R}$. Indeed, the map b must be $W^{1,p}(B_1)$, and the map c must be $L^p(B_1)$, $p > d$; such a requirement unlocks the uniform compactness of the solutions. Secondly, a proximity regime must be in force; that is, there must be $\overline{b}\in\mathbb{R}^d$ and $\overline{c}\in\mathbb{R}$ so that
\[
	\left\|b^i\,-\,\overline{b^i}\right\|_{W^{1,p}(B_1)}\,+\,\left\|c\,-\,\overline{c}\right\|_{L^\infty(B_1)}\,\ll\,1/2.
\]

\end{Remark}

\bibliography{biblio}
\bibliographystyle{plain}

\bigskip

\noindent\textsc{Raimundo Leit\~ao}\\
Universidade Federal Ceara\\
Department of Mathematics\\
CE-Brazil 60455-760\\
\noindent\texttt{rleitao@mat.ufc.br}

\bigskip

\noindent\textsc{Edgard A. Pimentel (Corresponding Author)}\\
Department of Mathematics\\
Pontifical Catholic University of Rio de Janeiro -- PUC-Rio\\
22451-900, G\'avea, Rio de Janeiro-RJ, Brazil\\
\noindent\texttt{pimentel@puc-rio.br}

\bigskip

\noindent\textsc{Makson S. Santos}\\
Department of Mathematics\\
Pontifical Catholic University of Rio de Janeiro -- PUC-Rio\\
22451-900, G\'avea, Rio de Janeiro-RJ, Brazil\\
\noindent\texttt{makson@mat.puc-rio.br}

\end{document}